\documentclass[reqno]{amsart}
\usepackage{amssymb, latexsym}
\usepackage{hyperref}

\let\Re=\undefined\DeclareMathOperator*{\Re}{Re}
\let\Im=\undefined\DeclareMathOperator*{\Im}{Im}

\def\C{{\mathbb C}} 
\def\R{{{\mathbb R}}}
\def\Z{{{\mathbb Z}}}

\newcommand{\eps}{{\varepsilon}}
\newcommand{\Tmax}{{T_{max}}}
\newcommand{\uhi}{{u_{hi}}}
\newcommand{\ulo}{{u_{lo}}}
\newcommand{\pl}{{P_{lo}}}
\newcommand{\ph}{{P_{hi}}}

\theoremstyle{plain}
\newtheorem{theorem}{Theorem}
\newtheorem{proposition}[theorem]{Proposition}
\newtheorem{lemma}[theorem]{Lemma}
\newtheorem{corollary}[theorem]{Corollary}
\theoremstyle{definition}

\newtheorem{definition}[theorem]{Definition}
\newtheorem{remark}[theorem]{Remark}

\numberwithin{equation}{section} \numberwithin{theorem}{section}

\begin{document}

\title[Global well-posedness for the cubic NLS in 4D]{Global well-posedness and scattering for the defocusing cubic NLS in four dimensions}
\author{Monica Vi\c{s}an}
\address{Department of Mathematics\\UCLA\\Los Angeles, CA 90095}
%\subjclass{35Q55} 

\vspace{-0.3in}
\begin{abstract}
In this short note we present a new proof of the global well-posedness and scattering result for the defocusing energy-critical NLS in four space dimensions obtained previously by Ryckman and Visan \cite{RV}.
The argument is inspired by the recent work of Dodson \cite{Dodson:d>2} on the mass-critical NLS.
\end{abstract}

\maketitle

\section{Introduction}

We study the initial-value problem for the defocusing cubic nonlinear Schr\"odinger equation in $\R\times\R^4:$
\begin{equation}\label{nls}
\begin{cases}
i u_t +\Delta u = |u|^2 u\\
u(0,x) = u_0(x)\in \dot H^1_x(\R^4),
\end{cases}
\end{equation}
where $u(t,x)$ is a complex-valued function on spacetime $\R_t\times \R^4_x$.

This equation is termed \emph{energy-critical} since the scaling symmetry 
\begin{equation}\label{scaling}
u(t,x)\mapsto u^\lambda (t,x) :=\lambda u(\lambda^2t, \lambda x),
\end{equation}
leaves invariant not only the class of solutions to \eqref{nls}, but also the (conserved) energy:
\begin{equation}\label{energy}
E(u(t))=\int_{\R^4} \frac{1}{2}|\nabla u (t,x)|^2 + \frac{1}{4}|u(t,x)|^4dx.
\end{equation}
Note that by Sobolev embedding, the energy space is precisely $\dot H^1_x(\R^4)$, which is also the space of the initial data.

Let us start by making the notion of solution more precise.

\begin{definition}[Solution]\label{D:solution} A function $u: I \times \R^4 \to \C$ on a non-empty time interval $0\in I \subset \R$ is a (strong) \emph{solution} to \eqref{nls} if it lies in the class $C^0_t \dot H^1_x(K \times \R^4) \cap L^{6}_{t,x}(K \times \R^4)$ for all compact $K \subset I$, and obeys the Duhamel formula
\begin{align}\label{old duhamel}
u(t) = e^{it\Delta} u(0) - i \int_{0}^{t} e^{i(t-s)\Delta}|u(s)|^2 u(s) \, ds
\end{align}
for all $t \in I$.  We refer to the interval $I$ as the \emph{lifespan} of $u$. We say that $u$ is a \emph{maximal-lifespan solution} if the solution cannot be extended to any strictly larger interval. We say that $u$ is a
\emph{global solution} if $I= \R$.
\end{definition}

Our main result is a new proof of the following:

\begin{theorem}[Global well-posedness and scattering]\label{T:main}
Let $u_0\in \dot H^1_x(\R^4)$.  Then there exists a unique global (strong) solution $u\in C_t^0\dot H^1_x(\R\times \R^4)$ to \eqref{nls}.  Moreover, this solution satisfies
$$
\int_{\R}\int_{\R^4} |u(t,x)|^6\, dx\, dt \leq C\bigl(\|u_0\|_{\dot H^1_x}\bigr).
$$
As a consequence of these spacetime bounds, the solution $u$ scatters, that is, there exist asymptotic states $u_\pm\in \dot H^1_x$ such that
\begin{align}\label{scattering}
\bigl\|u(t)-e^{it\Delta}u_\pm\bigr\|_{\dot H^1_x} \to 0 \quad \text{as}\quad t\to \pm\infty.
\end{align}
Furthermore, for any $u_+ \in \dot H^1_x$ $($or $u_- \in \dot H^1_x$$)$ there exists a unique global solution $u$ to \eqref{nls} such that \eqref{scattering} holds.
\end{theorem}

Theorem~\ref{T:main} was proved by Ryckman and Visan \cite{RV} building upon the monumental work of Colliander--Keel--Staffilani--Takaoka--Tao \cite{CKSTT:gwp} who established the equivalent result for the energy-critical nonlinear Schr\"odinger equation in three space dimensions.  For the analogous result in higher dimensions, see \cite{Berbec, thesis:art, Monica:thesis}.  For previous results for spherically-symmetric data, see \cite{borg:scatter, grillakis, tao:gwp radial}.

Recently, there has been a lot of activity towards proving the natural global well-posedness and scattering conjecture for the focusing energy-critical NLS.  This was made possible by the work of Kenig and Merle \cite{KM:NLS} who proved this conjecture in spatial dimensions $d=3,4,5$ for spherically-symmetric initial data.  Subsequently, Killip and Visan \cite{Berbec} established the focusing conjecture for arbitrary initial data and spatial dimensions
$d\geq 5$.  The focusing conjecture is still open in dimensions $d=3,4$ for arbitrary initial data.

In this short note we present a new proof of global well-posedness and scattering for \eqref{nls} inspired by the proof of the analogous theorem for the defocusing mass-critical NLS in dimensions $d\geq 3$ given by Dodson \cite{Dodson:d>2}; for previous results on the mass-critical NLS in both the defocusing and focusing cases with spherically-symmetric initial data, see \cite{KTV,KVZ, TVZ:sloth}.

We contend that the arguments in \cite{Berbec} offer much simplicity over those of \cite{thesis:art} (which builds upon \cite{CKSTT:gwp, RV}) and have also the advantage of treating both the defocusing and focusing cases for
$d\geq 5$.  As such, one is faced only with the task of simplifying, if possible, the arguments of  \cite{CKSTT:gwp, RV} in dimensions $d=3,4$.  The simple argument presented here does not cover the three dimensional case due to the absence of certain endpoint estimates.

The main reason for writing this paper is to demonstrate the applicability to the energy-critical setting of the new ideas introduced by Dodson \cite{Dodson:d>2}.  While we rely on other sources for several key ingredients in the proof, the combined argument is still shorter than that in \cite{RV} and, more importantly, it is more modular and hence easier to understand.  

\subsection{Outline of the proof}\label{SS:outline}

We will use the concentration-compactness approach and hence argue by contradiction.  As proven in \cite{KM:NLS, Berbec}, failure of Theorem~\ref{T:main} implies the existence of very special types of counterexamples.  Such counterexamples are then shown to have a wealth of properties not immediately apparent from their construction, so many properties in fact, that they cannot exist.

While we will make some further reductions later, the main property of the special counterexamples is almost periodicity modulo symmetries:

\begin{definition}[Almost periodicity modulo symmetries]\label{D:ap}
A solution $u$ to \eqref{nls} with lifespan $I$ is said to be \emph{almost periodic (modulo symmetries)} if there exist functions $N: I \to \R^+$, $x:I\to \R^4$, and $C: \R^+ \to \R^+$ such that for all $t \in I$ and $\eta > 0$,
$$
\int_{|x-x(t)| \geq C(\eta)/N(t)} \bigl|\nabla u(t,x)\bigr|^2\, dx + \int_{|\xi| \geq C(\eta) N(t)} |\xi|^{2}\, | \hat u(t,\xi)|^2\, d\xi\leq \eta .
$$
We refer to the function $N(t)$ as the \emph{frequency scale function} for the solution $u$, to $x(t)$ as the \emph{spatial center function}, and to $C(\eta)$ as the \emph{compactness modulus function}.
\end{definition}

\begin{remark} \label{R:small freq}
Almost periodic solutions obey the following: For each $\eta>0$ there exists $c(\eta)>0$ so that for all $t\in I$,
$$
\int_{|x-x(t)| \leq c(\eta)/N(t)} \bigl|\nabla u(t,x)\bigr|^2\, dx + \int_{|\xi| \leq c(\eta) N(t)} |\xi|^{2}\, | \hat u(t,\xi)|^2\, d\xi\leq \eta .
$$
\end{remark}

As shown in \cite[Corollary~3.6]{KTV} (see also \cite[Lemma 5.18]{ClayNotes}), the modulation parameters of almost periodic solutions obey a local constancy property.  In particular, we have

\begin{lemma}[Local constancy property]\label{L:local const}
Let $u:I\times\R^4\to \C$ be a maximal-lifespan almost periodic solution to \eqref{nls}.  Then there exists a small number $\delta$, depending only on $u$, such that if $t_0 \in I$ then
\begin{align*}
\bigl[t_0 - \delta N(t_0)^{-2}, t_0 + \delta N(t_0)^{-2}\bigr] \subset I
\end{align*}
and
\begin{align*}
N(t) \sim_u N(t_0) \quad \text{whenever} \quad |t-t_0| \leq \delta N(t_0)^{-2}.
\end{align*}
\end{lemma}

We recall next a consequence of the local constancy property; see \cite[Corollary~3.7]{KTV} and \cite[Corollary~5.19]{ClayNotes}.

\begin{corollary}[$N(t)$ at blowup]\label{C:blowup criterion}
Let $u:I\times\R^4\to \C$ be a maximal-lifespan almost periodic solution to \eqref{nls}.  If $T$ is any finite endpoint of $I$, then $N(t) \gtrsim_u |T-t|^{-1/2}$; in particular, $\lim_{t\to T} N(t)=\infty$.
\end{corollary}

Finally, we will need the following result linking the frequency scale function $N(t)$ of an almost periodic solution $u$ and its Strichartz norms: 

\begin{lemma}[Spacetime bounds]\label{L:ST-N(t)}
Let $u$ be an almost periodic solution to \eqref{nls} on a time interval $I$.  Then
\begin{align}\label{ST via N(t)}
\int_I N(t)^2\,dt \lesssim_u \|\nabla u\|_{L_t^2 L_x^4(I\times \R^4)}^2 \lesssim_u 1+ \int_I N(t)^2\,dt.
\end{align}
\end{lemma}

\begin{proof}
We recall that Lemma~5.21 in \cite{ClayNotes} shows that
$$
\int_I N(t)^2\,dt \lesssim_u \int_I\int_{\R^4} |u(t,x)|^6\, dx\, dt  \lesssim_u 1+ \int_I N(t)^2\,dt.
$$
The first inequality in \eqref{ST via N(t)} follows from the first inequality in the display above, Sobolev embedding, and interpolation with the energy norm, while the second inequality in \eqref{ST via N(t)} follows from the second inequality in the display above and an application of the Strichartz inequality.
\end{proof}

With these preliminaries out of the way, we can now describe the first major milestone in the proof of Theorem~\ref{T:main}:

\begin{theorem}[Reduction to almost periodic solutions, \cite{KM:NLS, Berbec}]\label{T:reduct}
Suppose that Theorem~\ref{T:main} failed.  Then there exists a maximal-lifespan solution $u:I\times\R^4\to \C$ to \eqref{nls} which is almost periodic and blows up both forward and backward in time in the sense that for all $t_0\in I$,
$$
\int_{t_0}^{\sup I}\int_{\R^4} |u(t,x)|^6\, dx\, dt=\int_{\inf I}^{t_0}\int_{\R^4} |u(t,x)|^6\, dx\, dt=\infty.
$$
\end{theorem}

The reduction to almost periodic solutions is by now a standard technique in the analysis of dispersive equations at critical regularity.  Their existence was first proved by Keraani \cite{keraani-l2} in the context of the mass-critical NLS and it was first used as a tool for proving global well-posedness by Kenig and Merle \cite{KM:NLS}.  As noted above, Theorem~\ref{T:reduct} was proved in \cite{KM:NLS, Berbec}; for other instances of the same techniques, see for example \cite{KenigMerle:wave, KenigMerle:H^1/2, KenigMerle:wave sup, KKSV, KSV, KTV, Berbec, ClayNotes, Miel, KV:wave, KV:wave radial, KVZ, TVZ:cc, TVZ:sloth}.

To continue, a simple rescaling argument (see, for example, the proof of Theorem~3.3 in \cite{TVZ:sloth}) allows us to restrict attention to almost periodic solutions that do not escape to arbitrarily low frequencies, at least on half of their maximal lifespan, say, on $[0, \Tmax)$.  Inspired by \cite{Dodson:d>2}, we further subdivide these into two classes depending on the control given by the interaction Morawetz inequality.  Finally, by
Lemma~\ref{L:local const}, we may subdivide $[0, \Tmax)$ into characteristic subintervals $J_k$ and set the frequency scale function $N(t)$ to be constant and equal to $N_k$ on each $J_k$.  Note that
$|J_k|=\delta N_k^{-2}$ and that we need to modify the compactness modulus function by a time-independent multiplicative factor.  Putting everything together, we obtain

\begin{theorem}[Two special scenarios for blowup]\label{T:enemies}
Suppose that Theorem \ref{T:main} failed.  Then there exists an almost periodic solution $u: [0, \Tmax)\times \R^4 \to \C$, such that
$$
\|u\|_{L^6_{t,x}( [0, \Tmax) \times \R^4)} =+\infty 
$$
and
$$
N(t)\equiv N_k \geq 1 \quad \text{for each} \quad t\in J_k \quad \text{where}\quad  [0, \Tmax)=\cup_k J_k.
$$
Furthermore,
$$
\text{either} \quad \int_0^{T_{\max}} N(t)^{-1}\, dt<\infty \quad\text{or}\quad \int_0^{T_{\max}}  N(t)^{-1}\, dt=\infty.
$$
\end{theorem}

Thus, in order to prove Theorem~\ref{T:main} we have to preclude the existence of the almost periodic solutions described in Theorem~\ref{T:enemies}.  The main ingredient in the proof will be a long-time
Strichartz inequality; see Theorem~\ref{T:LTS}.  This is modeled on the long-time Strichartz inequality for almost periodic solutions to the mass-critical NLS in dimensions $d\geq 3$ derived by
Dodson \cite{Dodson:d>2}.  The proof of Theorem~\ref{T:LTS} is by induction; the recurrence relation is derived using the Strichartz inequality combined with a bilinear Strichartz inequality (Lemma~\ref{L:bilinear Strichartz})
and a paraproduct estimate (Lemma~\ref{L:neg deriv}).

In Section~\ref{S:no cascade}, we preclude the rapid frequency-cascade scenario, that is, almost periodic solutions as in Theorem~\ref{T:enemies} for which $ \int_0^{T_{\max}} N(t)^{-1}\, dt<\infty$.
The proof has two ingredients:  The first is the long-time Strichartz inequality in Theorem~\ref{T:LTS}.  The second ingredient is the following 

\begin{proposition}[No-waste Duhamel formula, \cite{ ClayNotes,TVZ:cc}]\label{P:duhamel}
Let $u: [0, \Tmax)\times \R^4 \to \C$ be a solution as in Theorem~\ref{T:enemies}.  Then for all $t\in[0, \Tmax)$,
\begin{align*}
u(t) = i \lim_{T\to\, \Tmax} \int_t^T e^{i(t-s) \Delta}  |u(s)|^2u(s)  \, ds,
\end{align*}
where the limits are to be understood in the weak $\dot H^1_x$ topology.
\end{proposition}

Using Proposition~\ref{P:duhamel} and the Strichartz inequality, we upgrade the information given by Theorem~\ref{T:LTS} to obtain that the rapid frequency-cascade solution has finite mass.  In fact, we will prove that the mass must be zero, and hence the solution must be zero, which contradicts the fact that the solution has infinite spacetime norm.

In Section~\ref{S:no soliton}, we preclude the quasi-soliton scenario, that is, almost periodic solutions as in Theorem~\ref{T:enemies} for which $ \int_0^{T_{\max}} N(t)^{-1}\, dt=\infty$.
The main ingredient in the proof is a frequency-localized interaction Morawetz inequality; see Proposition~\ref{P:flim}.  To establish it, we will rely on Theorem~\ref{T:LTS} to control the error terms introduced by
frequency-localizing the usual interaction Morawetz inequality.  As the frequency-localized interaction Morawetz inequality yields uniform control (in terms of the frequency above which we localize) over
$ \int_I N(t)^{-1}\, dt$ for all compact time intervals $I\subset[0, \Tmax)$, we derive a contradiction with the fact that  $ \int_0^{T_{\max}} N(t)^{-1}\, dt=\infty$ by simply taking the interval $I$ to be sufficiently large
inside $[0, \Tmax)$.

\subsection*{Acknowledgements}
The author was supported in part by the Sloan Foundation and NSF grant DMS-0901166.  This work was completed while the author was a Harrington Faculty Fellow at the University of Texas at Austin.

%%%%%%%%%%%%%%%%%%%%%%%%%%%%%%%%%%%%%%%%%%%%%%%%%%%%%%%%%%%%%%%%%%%%%%%%%%%%%%%%%%%%%%%%%%%
%
%
%                                   Section
%
%
%%%%%%%%%%%%%%%%%%%%%%%%%%%%%%%%%%%%%%%%%%%%%%%%%%%%%%%%%%%%%%%%%%%%%%%%%%%%%%%%%%%%%%%%%%%

\section{Notation and useful lemmas}\label{S:notation}
We will often use the notation $X \lesssim Y$ whenever there exists some constant $C$ so that $X \leq CY$.  Similarly, we will use $X \sim Y$ if $X \lesssim Y \lesssim X$.  If $C$ depends upon some additional parameters, we will indicate this with subscripts; for example, $X \lesssim_u Y$ denotes the assertion that $X \leq C_u Y$ for some $C_u$ depending on $u$.

We define the Fourier transform on $\R^4$ to be
$$
\hat f(\xi) :=(2\pi)^{-2} \int_{\R^4} e^{- i x \cdot \xi} f(x) dx.
$$
We will make frequent use of the fractional differential/integral operators $|\nabla|^s$ defined by
$$
\widehat{|\nabla|^sf}(\xi) :=|\xi|^s \hat f (\xi).
$$
These define the homogeneous Sobolev norms
$$
\|f\|_{\dot H^s_x} := \| |\nabla|^s f \|_{L^2_x }.
$$

We will frequently denote the nonlinearity in \eqref{nls} by $F(u)$, that is, $F(u):=|u|^2u$.  We will use the notation $\O(X)$ to denote a quantity that resembles $X$, that is, a finite linear combination of terms that look like $X$, but possibly with some factors replaced by their complex conjugates.  For example, we will write
\begin{equation*}
F(u+v) = \sum_{j=0}^3 \O(u^j v^{3-j}).
\end{equation*}

We will also need some Littlewood--Paley theory.  Specifically, let $\varphi(\xi)$ be a smooth bump supported in the ball $|\xi| \leq 2$ and equalling one on the ball $|\xi| \leq 1$.  For each dyadic number
$N \in 2^\Z$ we define the Littlewood--Paley operators
\begin{align*}
\widehat{P_{\leq N}f}(\xi) :=  \varphi(\xi/N)\hat f (\xi), \qquad \widehat{P_{> N}f}(\xi) :=  (1-\varphi(\xi/N))\hat f (\xi),
\end{align*}
\begin{align*}
\widehat{P_N f}(\xi) :=  (\varphi(\xi/N) - \varphi (2 \xi /N))
\hat f (\xi).
\end{align*}
Similarly, we can define $P_{<N}$, $P_{\geq N}$, and $P_{M < \cdot \leq N} := P_{\leq N} - P_{\leq M}$, whenever $M$ and $N$ are dyadic numbers.  We will frequently write $f_{\leq N}$ for
$P_{\leq N} f$ and similarly for the other operators.

The Littlewood--Paley operators commute with derivative operators, the free propagator, and complex conjugation.  They are self-adjoint and bounded on every $L^p_x$ and $\dot H^s_x$ space for
$1 \leq p \leq\infty$ and $s\geq 0$.  They also obey the following Sobolev and Bernstein estimates:
\begin{align*}
\| |\nabla|^{\pm s} P_N f\|_{L^p_x} \sim N^{\pm s} \| P_N f \|_{L^p_x}, \qquad \|P_N f\|_{L^q_x} \lesssim N^{\frac{4}{p}-\frac{4}{q}} \| P_N f\|_{L^p_x},
\end{align*}
whenever $s \geq 0$ and $1 \leq p \leq q \leq \infty$.

We use $L^q_tL^r_x$ to denote the spacetime norm
$$
\|u\|_{L_t^qL_x^r} :=\Bigl(\int_{\R}\Bigl(\int_{\R^4} |u(t,x)|^r dx \Bigr)^{q/r} dt \Bigr)^{1/q},
$$
with the usual modifications when $q$ or $r$ is infinity, or when the domain $\R \times \R^4$ is replaced by some
smaller spacetime region.  When $q=r$ we abbreviate $L^q_tL^r_x$ by $L^q_{t,x}$.

Let $e^{it\Delta}$ be the free Schr\"odinger propagator.  In physical space this is given by the formula
$$
e^{it\Delta}f(x) = \frac{1}{(4 \pi i t)^2} \int_{\R^4} e^{i|x-y|^2/4t} f(y) dy.
$$
In particular, the propagator obeys the \emph{dispersive inequality}
\begin{equation}\label{dispersive ineq}
\|e^{it\Delta}f\|_{L^\infty_x(\R^4)} \lesssim |t|^{-2}\|f\|_{L^1_x(\R^4)}
\end{equation}
for all times $t\neq 0$.  As a consequence of these dispersive estimates, one obtains the Strichartz estimates; see, for example, \cite{gv:strichartz, tao:keel, Strichartz}.

\begin{lemma}[Strichartz inequality]\label{L:Strichartz}
Let $I$ be a compact time interval and let $u : I\times\R^4 \rightarrow \C$ be a solution to the forced Schr\"odinger equation
\begin{equation*}
i u_t + \Delta u = G
\end{equation*}
for some function $G$.  Then we have
\begin{equation}
\|\nabla u\|_{L_t^qL_x^r(I\times\R^4)} \lesssim \|u(t_0)\|_{\dot H^1_x (\R^4)} + \|\nabla G \|_{L^{\tilde q'}_t L^{\tilde r'}_x (I\times\R^4)}
\end{equation}
for any time $t_0 \in I$ and any exponents $(q,r)$ and $(\tilde q,\tilde r)$ obeying
$$
\tfrac1q+\tfrac2r=\tfrac1{\tilde q}+\tfrac2{\tilde r} =1 \quad \text{and}\quad 2\leq q, \tilde q\leq \infty.
$$
Here, as usual, $p'$ denotes the dual exponent to $p$, that is $1/p + 1/p' = 1$.
\end{lemma}

We also recall the bilinear Strichartz estimates.  These were introduced by Bourgain \cite{borg bil}; see also \cite{CKSTT:gwp, RV, Monica:thesis} for several extensions.  For the particular version we need,
see Corollary~4.19 in \cite{ClayNotes}.

\begin{lemma}[Bilinear Strichartz]\label{L:bilinear Strichartz}
For any spacetime slab $I \times \R^4$ and any frequencies $M>0$ and $N>0$,
\begin{align*}
\bigl\|u_{\leq M}v_{\geq N}\bigr\|_{L^2_{t,x} (I \times \R^4)} 
&\lesssim \frac{M^{1/2}}{N^{1/2}} \|\nabla u_{\leq M}\|_{S_0^*(I)}\| v_{\geq N}\|_{S_0^*(I)},
\end{align*}
where we use the notation
$$
\|u\|_{S^*_0(I)}:=\|u\|_{L_t^\infty L_x^2(I\times\R^4)} + \|(i \partial_t + \Delta)u\|_{L_{t,x}^{3/2}(I \times \R^4)}.
$$
\end{lemma}

Our next result is a paraproduct estimate; cf. the proof of Proposition~5.7 in \cite{RV}.  

\begin{lemma}[Paraproduct estimate]\label{L:neg deriv}
We have
\begin{align*}
\bigl\||\nabla|^{-2/3}(fg)\bigr\|_{L_x^{4/3}}   \lesssim \bigl\||\nabla|^{-2/3} f\bigr\|_{L_{x}^{p}}\bigl\||\nabla|^{2/3} g\bigr\|_{L_x^{q}},
\end{align*}
for any $1<p,q<\infty$ such that $\frac1p+\frac1q=\frac{11}{12}$.
\end{lemma}

\begin{proof}
The claim is equivalent to the following estimate
\begin{align*}
\bigl\||\nabla|^{-\frac{2}{3}}\{(|\nabla|^{\frac{2}{3}} f)(|\nabla|^{-\frac{2}{3}}g)\}\bigr\|_{L_x^{4/3}}   \lesssim \|f\|_{L_{x}^{p}}\|g\|_{L_x^{q}}, \quad \text{for } 1<p,q<\infty \text{ obeying } \tfrac1p+\tfrac1q=\tfrac{11}{12}.
\end{align*}
To prove this, we start by decomposing the left-hand side into $\pi_{l,h}$ and $\pi_{h,l}$, which represent the projections onto low-high and high-low frequency interactions.
More precisely, for any pair of functions $(\phi, \psi)$, we write
$$
\pi_{l,h}(\phi, \psi):=\sum_{N\lesssim M}\phi_N \psi_M \quad \text{and} \quad \pi_{h,l}(\phi, \psi):=\sum_{N\gg M}\phi_N \psi_M.
$$

Let us consider first the low-high interactions.  By Sobolev embedding,
\begin{align*}
\bigl\||\nabla|^{-\frac{2}{3}}\pi_{l,h}(|\nabla|^{\frac{2}{3}} f, |\nabla|^{-\frac{2}{3}}g)\bigr\|_{L_x^{4/3}}  \lesssim \bigl\|\pi_{l,h}(|\nabla|^{\frac{2}{3}}f,|\nabla|^{-\frac{2}{3}} g)\bigr\|_{L_x^{12/11}}. 
\end{align*}
Now we only have to observe that the multiplier associated to the operator
$$
T(f,g)=\pi_{l,h}(|\nabla|^{\frac{2}{3}} f, |\nabla|^{-\frac{2}{3}}g),   \quad \text{that is,}\quad   \sum_{N\lesssim M} |\xi_1|^{\frac{2}{3}}\widehat{P_N f}(\xi_1) |\xi_2|^{-\frac{2}{3}}\widehat{P_M g}(\xi_2),
$$
is a symbol of order zero with $\xi = (\xi_1,\xi_2)$, since then a theorem of Coifman and Meyer (see, for example, \cite{coifmey:1,coifmey:2}) will conclude our claim.

To deal with the $\pi_{h,l}$ term, we first notice that the multiplier associated to the operator $T(f,h)=|\nabla|^{-\frac{2}{3}} \pi_{h,l}(|\nabla|^{\frac{2}{3}} f, h)$, that is,
$$
\sum_{N\gg M}|\xi_1+\xi_2|^{-\frac{2}{3}}|\xi_1|^{\frac{2}{3}}\widehat{P_N f}(\xi_1) \widehat{P_M h}(\xi_2),
$$
is a symbol of order zero. The result cited above then yields
$$
\bigl\||\nabla|^{-\frac{2}{3}}\pi_{h,l}(|\nabla|^{\frac{2}{3}} f, |\nabla|^{-\frac{2}{3}}g)\bigr\|_{L_x^{4/3}}  \lesssim \|f\|_{L^p_x} \||\nabla|^{-\frac{2}{3}}g\|_{L_x^r},
$$
where $r$ is such that $\frac1p+\frac1r=\frac34$.  The claim now follows by applying Sobolev embedding to the second factor on the right-hand side of the inequality above.
\end{proof}

Whenever we will employ the paraproduct estimate above, we will also use the product rule for fractional derivatives:

\begin{lemma}[Product rule, \cite{ChristWeinstein}]\label{L:product rule}
Let $s\in(0,1]$ and $1<p, p_1,p_2,q_1,q_2<\infty$ such that $\frac 1p =\frac 1{p_i}+\frac 1{q_i}$ for $i=1,2$.  Then,
$$
\bigl\||\nabla|^s(fg)\bigr\|_{L^p} \lesssim \bigl\|f\bigr\|_{L^{p_1}}\bigl\||\nabla|^s g\bigr\|_{L^{q_1}} + \bigl\||\nabla|^s f\bigr\|_{L^{p_2}}\bigl\|g\bigr\|_{L^{q_2}}.
$$
\end{lemma}

%%%%%%%%%%%%%%%%%%%%%%%%%%%%%%%%%%%%%%%%%%%%%%%%%%%%%%%%%%%%%%%%%%%%%%%%%%%%%%%%%%%%%%%%%%%
%
%
%                                   Section
%
%
%%%%%%%%%%%%%%%%%%%%%%%%%%%%%%%%%%%%%%%%%%%%%%%%%%%%%%%%%%%%%%%%%%%%%%%%%%%%%%%%%%%%%%%%%%%

\section{Long-time Strichartz estimates}\label{S:LTS}

The main result of this section is a long-time Strichartz estimate.  This is inspired by an analogous statement for the mass-critical NLS obtained by Dodson \cite{Dodson:d>2}.

\begin{theorem}[Long-time Strichartz estimates]\label{T:LTS}
Let $u:[0, \Tmax)\times\R^4\to \C$ be an almost periodic solution to \eqref{nls} with $N(t)\equiv N_k \geq 1$ on each characteristic interval $J_k\subset [0, \Tmax)$.  Then, on any compact time interval
$I\subset [0, \Tmax)$, which is a union of contiguous intervals $J_k$, and for any frequency $N>0$,
\begin{align}\label{E:finite LTS}
\|\nabla u_{\leq N}\|_{L_t^2 L_x^4(I\times\R^4)} \lesssim_u 1 + N^{3/2} K^{1/2},
\end{align}
where $K:=\int_I N(t)^{-1}\, dt$.  Moreover, for any $\eta>0$ there exists $N_0=N_0(\eta)>0$ such that for all $N\leq N_0$,
\begin{align}\label{E:small LTS}
\|\nabla u_{\leq N}\|_{L_t^2 L_x^4(I\times\R^4)} \lesssim_u \eta\bigl(1 + N^{3/2} K^{1/2}\bigr).
\end{align}
Importantly, the constant $N_0$ and the implicit constants in \eqref{E:finite LTS} and \eqref{E:small LTS} are independent of the interval $I$.
\end{theorem}

\begin{proof}
Fix a compact time interval $I\subset [0, \Tmax)$, which is a union of contiguous intervals $J_k$.  Throughout the proof all spacetime norms will be on $I\times\R^4$, unless we specify otherwise. Let $\eta_0>0$ be a small parameter to be chosen later.  By Remark~\ref{R:small freq}, there exists $c_0=c_0(\eta_0)$ such that
\begin{align}\label{c_0}
\|\nabla u_{\leq c_0 N(t)}\|_{L_t^\infty L^2_x}\leq \eta_0.             
\end{align}

For $N>0$ we define
$$
A(N):=\|\nabla u_{\leq N}\|_{L_t^2 L_x^4(I\times\R^4)}.
$$
Note that Lemma~\ref{L:ST-N(t)} implies
\begin{align}\label{E:N large}
A(N)\lesssim_u 1 + N^{3/2} K^{1/2} \quad \text{whenever}\quad N\geq \Biggl( \frac{\int_I N(t)^2\, dt}{\int_I N(t)^{-1}\, dt} \Biggr)^{1/3},
\end{align}
and, in particular, whenever $N\geq N_{max}:= \sup_{t\in I}N(t)$.  We will obtain the result for arbitrary frequencies $N>0$ by induction.  Our first step is to obtain a recurrence relation for $A(N)$.
We start with an application of the Strichartz inequality:
\begin{align}\label{A est}
A(N)\lesssim \inf_{t\in I}\|\nabla u_{\leq N}(t)\|_{L^2_x} + \bigl\|  \nabla P_{\leq N} F(u) \bigr\|_{L_t^2 L_x^{4/3}}.
\end{align}
We decompose $u=u_{\leq N/\eta_0} + u_{>N/\eta_0}$ and then further decompose $u(t)=u_{\leq c_0 N(t)}(t) + u_{>c_0N(t)}(t)$ to obtain
\begin{align}\label{decomp}
\nabla F(u)&=\nabla F\bigl( u_{> N/\eta_0}\bigr) + \O\bigl(\nabla  u_{> N/\eta_0} u_{\leq N/\eta_0} u_{\leq c_0N(t)}\bigr) + \O\bigl(\nabla  u_{\leq N/\eta_0} u_{\leq c_0N(t)}^2\bigr) \notag\\
&\quad + \O\bigl(\nabla  u_{> N/\eta_0} u_{\leq N/\eta_0} u_{> c_0N(t)}\bigr)+ \O\bigl(\nabla  u_{\leq N/\eta_0} u_{> c_0N(t)}^2\bigr).
\end{align}
Next, we will estimate the contributions of each of these terms to \eqref{A est}.

To estimate the contribution of the first term on the right-hand side of \eqref{decomp} to \eqref{A est} we use the Bernstein inequality followed by Lemma~\ref{L:neg deriv}, Lemma~\ref{L:product rule}, H\"older, and Sobolev embedding:
\begin{align}\label{1}
\bigl\|  \nabla P_{\leq N} F\bigl( u_{> N/\eta_0}\bigr) \bigr\|_{L_t^2 L_x^{4/3}}
&\lesssim N^{5/3} \bigl\|  |\nabla|^{-2/3} F\bigl( u_{> N/\eta_0}\bigr) \bigr\|_{L_t^2 L_x^{4/3}}\notag\\
&\lesssim N^{5/3} \bigl\|  |\nabla|^{-2/3} u_{> N/\eta_0} \bigr\|_{L_t^2 L_x^4} \bigl\|  |\nabla|^{2/3}\O\bigl( u^2_{> N/\eta_0}\bigr) \bigr\|_{L_t^\infty L_x^{3/2}}\notag\\
&\lesssim  N^{5/3} \bigl\|  |\nabla|^{-2/3} u_{> N/\eta_0} \bigr\|_{L_t^2 L_x^4}  \bigl\|  |\nabla|^{2/3} u_{> N/\eta_0} \bigr\|_{L_t^\infty L_x^{12/5}} \| u_{> N/\eta_0} \|_{L_t^\infty L_x^{4}}\notag\\
&\lesssim  N^{5/3} \bigl\|  |\nabla|^{-2/3} u_{> N/\eta_0} \bigr\|_{L_t^2 L_x^4} \|u\|_{L_t^\infty \dot H^1_x}^2\notag\\
&\lesssim_u \sum_{M> N/\eta_0} \Bigl( \frac NM \Bigr)^{5/3} A(M).
\end{align}

We turn now to the contribution to \eqref{A est} of the second term on the right-hand side of \eqref{decomp}.  Employing again Bernstein's inequality, Lemma~\ref{L:neg deriv}, H\"older, Sobolev embedding, and \eqref{c_0},
we obtain
\begin{align}\label{2}
\bigl\|  P_{\leq N}  &\O\bigl(\nabla  u_{> N/\eta_0} u_{\leq N/\eta_0} u_{\leq c_0N(t)}\bigr)  \bigr\|_{L_t^2 L_x^{4/3}}\notag\\
&\lesssim N^{2/3} \bigl\| |\nabla|^{-2/3}  \O\bigl(\nabla  u_{> N/\eta_0} u_{\leq N/\eta_0} u_{\leq c_0N(t)}\bigr)  \bigr\|_{L_t^2 L_x^{4/3}}\notag\\
&\lesssim N^{2/3} \bigl\| |\nabla|^{2/3}  u_{\leq N/\eta_0}\bigr\|_{L_t^2 L_x^6} \bigl\| |\nabla|^{-2/3} \O\bigl(\nabla  u_{> N/\eta_0} u_{\leq c_0N(t)}\bigr)  \bigr\|_{L_t^\infty L_x^{4/3}}\notag\\
&\lesssim N^{2/3} \| \nabla u_{\leq N/\eta_0}\|_{L_t^2 L_x^4}  \bigl\| |\nabla|^{2/3}  u_{\leq c_0N(t)}\bigl\|_{L_t^\infty L_x^{12/5}} \bigl\| |\nabla|^{1/3}  u_{> N/\eta_0} \bigl\|_{L_t^\infty L_x^2} \notag\\
&\lesssim_u N^{2/3} A\bigl(N/\eta_0\bigr) \eta_0  (N/\eta_0)^{-2/3}\notag\\
&\lesssim_u \eta_0^{5/3} A\bigl(N/\eta_0\bigr).
\end{align}

To estimate the contribution of the third term on the right-hand side of \eqref{decomp}, we use H\"older's inequality, Sobolev embedding, and \eqref{c_0}:
\begin{align}\label{3}
\bigl\|  P_{\leq N}  \O\bigl(\nabla  u_{\leq N/\eta_0} u_{\leq c_0N(t)}^2\bigr)  \bigr\|_{L_t^2 L_x^{4/3}}
&\lesssim  \| \nabla u_{\leq N/\eta_0}\|_{L_t^2 L_x^4} \| u_{\leq c_0N(t)}\|^2_{L_t^\infty L_x^4}
\lesssim \eta_0^2A\bigl(N/\eta_0\bigr).
\end{align} 

We consider next the contribution of the fourth term on the right-hand side of \eqref{decomp}.  By Bernstein and then H\"older,
\begin{align*}
\bigl\|  P_{\leq N}  \O\bigl(\nabla  u_{> N/\eta_0} u_{\leq N/\eta_0} u_{> c_0N(t)}\bigr)  \bigr\|_{L_t^2 L_x^{4/3}}
&\lesssim N\bigl\| \nabla  u_{> N/\eta_0} u_{\leq N/\eta_0} u_{> c_0N(t)}\bigr\|_{L_t^2 L_x^1}\\
&\lesssim N\| \nabla  u_{> N/\eta_0} \|_{L_t^\infty L_x^2} \| u_{\leq N/\eta_0} u_{> c_0N(t)} \|_{L_{t,x}^2}\\
&\lesssim_u N\| u_{\leq N/\eta_0} u_{> c_0N(t)}  \|_{L_{t,x}^2}.
\end{align*}
To continue, we use the decomposition of the time interval $I$ into subintervals $J_k$ where $N(t)$ is constant and apply the bilinear Strichartz estimate Lemma~\ref{L:bilinear Strichartz} on each of these subintervals.
Note that by Lemma~\ref{L:ST-N(t)} and H\"older's inequality, on each $J_k$ we have
$$
\|\nabla u\|_{L_t^2L_x^4(J_k\times\R^4)}+\|\nabla F(u)\|_{L_{t,x}^{3/2}(J_k\times\R^4)} \lesssim_u 1 \quad \text{and hence }\quad \|\nabla u\|_{S_0^*(J_k)}\lesssim_u 1.
$$
Thus, using also Bernstein's inequality,
\begin{align*}
 \| u_{\leq N/\eta_0} u_{> c_0N(t)}  \bigr\|_{L_{t,x}^2(J_k\times\R^4)}
&\lesssim \frac{(N/\eta_0)^{1/2}}{(c_0N_k)^{1/2}} \|\nabla u_{\leq  N/\eta_0}\|_{S_0^*(J_k)}\|u_{> c_0N_k}\|_{S_0^*(J_k)}\\
&\lesssim_u \frac {N^{1/2}}{\eta_0^{1/2}c_0^{3/2}N_k^{3/2}}  \|\nabla u_{\leq  N/\eta_0}\|_{S_0^*(J_k)}.
\end{align*}
Let us remark that the term $ \|\nabla u_{\leq  N/\eta_0}\|_{S_0^*(J_k)}$ will be source of one of the small parameters $\eta$ in claim \eqref{E:small LTS} and this is why we choose to keep it.  Summing the estimates above over the subintervals $J_k$ and invoking again the local constancy property Lemma~\ref{L:local const}, we find
\begin{align}\label{E:bilinear}
\| u_{\leq N/\eta_0} u_{> c_0N(t)}  \bigr\|_{L_{t,x}^2(I\times\R^4)}
&\lesssim_u  \frac{ N^{1/2}}{\eta_0^{1/2}c_0^{3/2}} \Bigl(\sum_{J_k\subset I} \frac1{N_k^3}\Bigr)^{1/2} \sup_{J_k\subset I}  \|\nabla u_{\leq  N/\eta_0}\|_{S_0^*(J_k)}\notag\\
&\lesssim_u  \frac{N^{1/2} K^{1/2}}{\eta_0^{1/2}c_0^{3/2}}  \sup_{J_k\subset I}  \|\nabla u_{\leq  N/\eta_0}\|_{S_0^*(J_k)}.
\end{align}
Thus, the contribution of the fourth term on the right-hand side of \eqref{decomp} can be bounded as follows:
\begin{align}\label{4}
\bigl\|  P_{\leq N}  \O\bigl(\nabla  u_{> N/\eta_0} u_{\leq N/\eta_0}& u_{> c_0N(t)}\bigr)  \bigr\|_{L_t^2 L_x^{4/3}} 
\lesssim_u  \frac{N^{3/2} K^{1/2}}{\eta_0^{1/2}c_0^{3/2}} \sup_{J_k\subset I}  \|\nabla u_{\leq  N/\eta_0}\|_{S_0^*(J_k)}.
\end{align}

We are left with the contribution to \eqref{A est} of the last term on the right-hand side of \eqref{decomp}.  Using the H\"older inequality combined with the arguments used to establish \eqref{E:bilinear} and one more application of the Bernstein inequality, we find
\begin{align}\label{5}
\bigl\|  P_{\leq N} \O\bigl(\nabla  u_{\leq N/\eta_0} u_{> c_0N(t)}^2\bigr)\bigr\|_{L_t^2 L_x^{4/3}}
&\lesssim \| u_{> c_0N(t)} \|_{L_t^\infty L_x^4} \| \nabla u_{\leq N/\eta_0} u_{> c_0N(t)} \|_{L_{t,x}^2}\notag\\
&\lesssim_u  \frac{N^{1/2} K^{1/2}}{\eta_0^{1/2}c_0^{3/2}} \sup_{J_k\subset I}  \|\Delta u_{\leq  N/\eta_0}\|_{S_0^*(J_k)}\notag\\
&\lesssim_u  \frac{N^{3/2} K^{1/2}}{\eta_0^{3/2}c_0^{3/2}} \sup_{J_k\subset I}  \|\nabla u_{\leq  N/\eta_0}\|_{S_0^*(J_k)}.
\end{align}

Putting everything together, we obtain
\begin{align}\label{A est all}
A(N)
&\lesssim_u \inf_{t\in I}\|\nabla u_{\leq N}(t)\|_{L^2_x} + \frac{N^{3/2} K^{1/2}}{\eta_0^{3/2}c_0^{3/2}} \sup_{J_k\subset I}  \|\nabla u_{\leq  N/\eta_0}\|_{S_0^*(J_k)}
	+  \sum_{M> \frac{N}{\eta_0}} \Bigl( \frac NM \Bigr)^{5/3} A(M).
\end{align}
The inductive step in the proof of claims \eqref{E:finite LTS} and \eqref{E:small LTS} will rely on this recurrence relation.

Let us first address \eqref{E:finite LTS}.  Recall that by \eqref{E:N large}, the claim holds for $N\geq N_{max}$, that is,
\begin{align}\label{base step}
A(N)\leq  C(u) \bigl[1+N^{3/2}K^{1/2}\bigr],
\end{align}
for some constant $C(u)>0$ and all $N\geq N_{max}$.  Rewriting \eqref{A est all} as
\begin{align}\label{A est all 1}
A(N)&\leq \tilde C(u)\Bigl\{1+ \frac{N^{3/2} K^{1/2}}{\eta_0^{3/2}c_0^{3/2}}+  \sum_{M> \frac{N}{\eta_0}} \Bigl( \frac NM \Bigr)^{5/3} A(M)\Bigr\},
\end{align}
we can inductively prove the claim by halving the frequency $N$ at each step.  For example, assuming that \eqref{base step} holds for frequencies larger or equal to $N$, an application of \eqref{A est all 1}
(with $\eta_0\leq 1/2$) yields
\begin{align*}
A\bigl(N/2\bigr)
&\leq  \tilde C(u)\Bigl\{1+ \frac{(N/2)^{3/2} K^{1/2}}{\eta_0^{3/2}c_0^{3/2}}+  C(u) \sum_{M> \frac{N}{2\eta_0}} \Bigl( \frac N{2M} \Bigr)^{5/3}  \bigl[1+M^{3/2}K^{1/2}\bigr] \Bigr\}\\
&\leq \tilde C(u)\Bigl\{1+ \frac{(N/2)^{3/2} K^{1/2}}{\eta_0^{3/2}c_0^{3/2}}+  \eta_0^{5/3}C(u) + \eta_0^{1/6} C(u) (N/2)^{3/2}K^{1/2}\Bigr\}.
\end{align*}
Choosing $\eta_0=\eta_0(u)$ small enough so that $ \eta_0^{1/6}\tilde C(u)\leq 1/2$, we thus obtain
$$
A\bigl(N/2\bigr)\leq   \frac12 C(u) \Bigl\{1+(N/2)^{3/2}K^{1/2}\Bigr\} + \tilde C(u)\Bigl\{1 + \frac{(N/2)^{3/2} K^{1/2}}{\eta_0^{3/2}c_0^{3/2}}\Bigr\}.
$$
The claim now follows by setting $C(u)\geq 2\tilde C(u)\eta_0^{-3/2}c_0^{-3/2}$.

Next we turn to \eqref{E:small LTS}.  To exhibit the small constant $\eta$, we will need the following

\begin{lemma}[Vanishing of the small frequencies]\label{L:small freq}
Under the assumptions of Theorem~\ref{T:LTS}, we have
$$
f(N):= \|\nabla u_{\leq N}\|_{L_t^\infty L_x^2([0, \Tmax)\times\R^4)} + \sup_{J_k\subset [0, \Tmax)}\|\nabla u_{\leq N}\|_{S_0^*(J_k)} \to 0 \quad \text{as} \quad N\to 0.
$$
\end{lemma}
 
\begin{proof}
As by hypothesis $\inf_{t\in [0, \Tmax)}N(t)\geq 1$, Remark~\ref{R:small freq} yields
\begin{align}\label{E:energy to 0}
\lim_{N\to 0} \|\nabla u_{\leq N}\|_{L_t^\infty L_x^2([0, \Tmax)\times\R^4)}=0.
\end{align}

Now fix a characteristic interval $J_k\subset [0, \Tmax)$ and recall that all Strichartz norms of $u$ are bounded on $J_k$; cf. Lemma~\ref{L:ST-N(t)}.  In particular, we have
\begin{align*}
\|\nabla u\|_{L_t^2L_x^4(J_k\times\R^4)} + \|u\|_{L_t^3L_x^{12}(J_k\times\R^4)} + \| u\|_{L_{t,x}^6(J_k\times\R^4)}\lesssim_u 1.
\end{align*}
Using this followed by the decomposition $u= u_{\leq N^{1/2}}+u_{> N^{1/2}}$, H\"older, and Bernstein, for any frequency $N>0$ we estimate
\begin{align*}
\|\nabla u_{\leq N}\|_{S_0^*(J_k)}
&=  \|\nabla u_{\leq N}\|_{L_t^\infty L_x^2} + \|\nabla P_{\leq N} F(u)\|_{L_{t,x}^{3/2}}\\
&\lesssim \|\nabla u_{\leq N}\|_{L_t^\infty L_x^2} + \|\nabla P_{\leq N} F(u_{> N^{1/2}})\|_{L_{t,x}^{3/2}} + \|\nabla u_{> N^{1/2}} u_{\leq N^{1/2}} u\|_{L_{t,x}^{3/2}}\\
&\quad + \|\nabla u_{\leq N^{1/2}} u^2\|_{L_{t,x}^{3/2}} \\
&\lesssim \|\nabla u_{\leq N}\|_{L_t^\infty L_x^2} + N \|u_{> N^{1/2}}\|_{L_t^2L_x^4}\|u_{> N^{1/2}}\|_{L_{t,x}^6}\|u_{> N^{1/2}}\|_{L_t^\infty L_x^4}\\
&\quad +  \|\nabla u_{> N^{1/2}}\|_{L_t^2L_x^4} \|u_{\leq N^{1/2}}\|_{L_t^\infty L_x^4} \|u\|_{L_{t,x}^6} +  \|\nabla u_{\leq N^{1/2}}\|_{L_t^\infty L_x^2}\|u\|_{L_t^3L_x^{12}}^2\\
&\lesssim_u \|\nabla u_{\leq N}\|_{L_t^\infty L_x^2} + N^{1/2} + \|\nabla u_{\leq N^{1/2}}\|_{L_t^\infty L_x^2}.
\end{align*}
All spacetime norms in the estimates above are on $J_k\times\R^4$.  As $J_k\subset [0, \Tmax)$ was arbitrary, we find
$$
\sup_{J_k\subset [0, \Tmax)}\|\nabla u_{\leq N}\|_{S_0^*(J_k)}\lesssim_u  N^{1/2}+  \|\nabla u_{\leq N}\|_{L_t^\infty L_x^2([0, \Tmax)\times\R^4)}  + \|\nabla u_{\leq N^{1/2}}\|_{L_t^\infty L_x^2([0, \Tmax)\times\R^4)}.
$$
The claim now follows by combining this with \eqref{E:energy to 0}.
\end{proof}

We are now ready to prove \eqref{E:small LTS}.  Using \eqref{E:finite LTS} and Lemma~\ref{L:small freq}, the estimate \eqref{A est all} implies
\begin{align*}
A(N)&\lesssim_u f(N) + \frac{N^{3/2} K^{1/2}}{\eta_0^{3/2}c_0^{3/2}}  f(N) +  \sum_{M> \frac{N}{\eta_0}} \Bigl( \frac NM \Bigr)^{5/3} A(M)\\
&\lesssim_u f(N) + \eta_0^{5/3} + \Biggl\{ \frac{f(N)}{\eta_0^{3/2}c_0^{3/2}} + \eta_0^{1/6} \Biggr\} N^{3/2}K^{1/2}.
\end{align*}
Thus, for any $\eta>0$, choosing first $\eta_0=\eta_0(\eta)$ such that $\eta_0^{1/6}\leq \eta$ and then $N_0=N_0(\eta)$ such that $\frac{f(N_0)}{\eta_0^{3/2}c_0^{3/2}}\leq \eta$, we obtain
$$
A(N)\lesssim_u \eta \bigl(1+ N^{3/2}K^{1/2}\bigr) \quad \text{for all} \quad N\leq N_0.
$$
This completes the proof of Theorem~\ref{T:LTS}.
\end{proof}

The following consequence of Theorem~\ref{T:LTS} will be useful in Section~\ref{S:no soliton}.

\begin{corollary}[Low and high frequencies control]\label{C:LTS}
Let $u:[0, \Tmax)\times\R^4\to \C$ be an almost periodic solution to \eqref{nls} with $N(t)\equiv N_k \geq 1$ on each characteristic interval $J_k\subset [0, \Tmax)$.  Then, on any compact time interval
$I\subset [0, \Tmax)$, which is a union of contiguous subintervals $J_k$, and for any frequency $N>0$,
\begin{align}\label{LTS:high}
\|u_{\geq N}\|_{L_t^q L_x^r (I\times\R^4)} \lesssim_u N^{-1} (1+N^3K)^{1/q} \quad \text{for all} \quad \tfrac1q+\tfrac2r=1 \text{ with } 3<q\leq \infty.
\end{align} 
Moreover, for any $\eta>0$ there exists $N_0=N_0(\eta)$ such that for all $N\leq N_0$ we have
\begin{align}\label{LTS:low}
\|\nabla u_{\leq N}\|_{L_t^q L_x^r (I\times\R^4)} \lesssim_u \eta(1+N^3K)^{1/q} \quad \text{for all} \quad \tfrac1q+\tfrac2r=1 \text{ with } 2\leq q\leq \infty.
\end{align} 
The constant $N_0$ and the implicit constants in \eqref{LTS:high} and \eqref{LTS:low} are independent of the interval $I$.
\end{corollary}

\begin{proof}
We first address \eqref{LTS:high}.  By \eqref{E:finite LTS} and Bernstein's inequality, for any $\eps>0$ and any frequency $N>0$ we have
\begin{align*}
\bigl\||\nabla|^{-1/2-\eps}u_{\geq N}\bigr \|_{L_t^2L_x^4(I\times\R^4)} 
&\lesssim \sum_{M\geq N} M^{-3/2-\eps} \|\nabla u_M \|_{L_t^2L_x^4(I\times\R^4)} \\
&\lesssim_u \sum_{M\geq N} M^{-3/2-\eps} (1+M^{3/2}K^{1/2})\\
&\lesssim_u N^{-3/2-\eps} (1+N^3K)^{1/2}.
\end{align*}
The claim now follows by interpolating with the energy bound:
\begin{align*}
\|u_{\geq N}\|_{L_t^q L_x^r (I\times\R^4)}
&\lesssim \bigl\||\nabla|^{-\frac12-\frac{q-3}2}u_{\geq N} \bigr\|_{L_t^2L_x^4(I\times\R^4)}^{2/q} \|\nabla u_{\geq N} \|_{L_t^\infty L_x^2(I\times\R^4)}^{1-2/q} 
\lesssim_u N^{-1} (1+N^3K)^{1/q},
\end{align*}
whenever $\frac1q+\frac2r=1$ and $3<q\leq \infty$.

We turn now to \eqref{LTS:low}.  As $\inf_{t\in I} N(t)\geq 1$, Remark~\ref{R:small freq} yields that for any $\eta>0$ there exists $N_0(\eta)$ such that for all $N\leq N_0$,
$$
\|\nabla u_{\leq N}\|_{L_t^\infty L_x^2(I\times\R^4)}\leq \eta.
$$
The claim follows by interpolating with \eqref{E:small LTS}.
\end{proof}

%%%%%%%%%%%%%%%%%%%%%%%%%%%%%%%%%%%%%%%%%%%%%%%%%%%%%%%%%%%%%%%%%%%%%%%%%%%%%%%%%%%%%%%%%%%
%
%
%                                   Section
%
%
%%%%%%%%%%%%%%%%%%%%%%%%%%%%%%%%%%%%%%%%%%%%%%%%%%%%%%%%%%%%%%%%%%%%%%%%%%%%%%%%%%%%%%%%%%%

\section{The rapid frequency-cascade scenario}\label{S:no cascade}

In this section, we preclude the existence of almost periodic solutions as in Theorem~\ref{T:enemies} for which $\int_0^{\Tmax}N(t)^{-1}\,dt<\infty$.  We will show their existence is inconsistent with the conservation of mass.

\begin{theorem}[No rapid frequency-cascades]\label{T:no cascade}
There are no almost periodic solutions $u:[0, \Tmax)\times\R^4\to \C$ to \eqref{nls} with $N(t)\equiv N_k \geq 1$ on each characteristic interval $J_k\subset [0, \Tmax)$
such that $\|u\|_{L^6_{t,x}( [0, \Tmax) \times \R^4)} =+\infty$ and
\begin{align}\label{finite int}
\int_0^{\Tmax}N(t)^{-1}\,dt<\infty.
\end{align}
\end{theorem}

\begin{proof}
We argue by contradiction.  Let $u$ be such a solution.  Then by Corollary~\ref{C:blowup criterion},
\begin{align}\label{N blowup}
\lim_{t\to \Tmax} N(t)=\infty,
\end{align}
whether $\Tmax$ is finite or infinite.  Thus, by Remark~\ref{R:small freq} we have
\begin{align}\label{vanishing}
\lim_{t\to \Tmax} \|\nabla u_{\leq N} (t)\|_{L_x^2}=0 \quad \text{for any} \quad N>0.
\end{align}

Now let $I_n$ be a nested sequence of compact subintervals of $[0, \Tmax)$ that are unions of contiguous characteristic subintervals $J_k$.  On each $I_n$ we may now apply Theorem~\ref{T:LTS}.
Specifically, using \eqref{A est all} together with the hypothesis \eqref{finite int}, we derive
\begin{align*}
A_n(N)&:= \|\nabla u_{\leq N}\|_{L_t^2L_x^4(I_n\times\R^4)}\\
&\lesssim_u \inf_{t\in I_n}\! \|\nabla u_{\leq N} (t)\|_{L_x^2} + \frac{N^{3/2}}{\eta_0^{3/2}c_0^{3/2}} \Bigl[\int_0^{\Tmax}N(t)^{-1}\,dt\Bigr]^{1/2}\!\!+\!\!  \sum_{M> \frac{N}{\eta_0}} \!\!\Bigl( \frac NM \Bigr)^{5/3} \! A_n(M)\\
&\lesssim_u \inf_{t\in I_n}\! \|\nabla u_{\leq N} (t)\|_{L_x^2}  + \frac{N^{3/2}}{\eta_0^{3/2}c_0^{3/2}} + \sum_{M> \frac{N}{\eta_0}} \Bigl( \frac NM \Bigr)^{5/3} A_n(M)
\end{align*}
for all frequencies $N>0$.  Arguing as for \eqref{E:finite LTS}, we find
$$
\|\nabla u_{\leq N}\|_{L_t^2L_x^4(I_n\times\R^4)}\lesssim_u \inf_{t\in I_n} \|\nabla u_{\leq N} (t)\|_{L_x^2} + N^{3/2} \quad \text{for all} \quad N>0.
$$
Letting $n$ tend to infinity and invoking \eqref{vanishing}, we obtain
\begin{align}\label{good Strich}
\|\nabla u_{\leq N}\|_{L_t^2L_x^4([0, \Tmax)\times\R^4)}\lesssim_u N^{3/2} \quad \text{for all} \quad N>0.
\end{align}

Our next claim is that \eqref{good Strich} implies
\begin{align}\label{good energy}
\|\nabla u_{\leq N}\|_{L_t^\infty L_x^2([0, \Tmax)\times\R^4)}\lesssim_u N^{3/2} \quad \text{for all} \quad N>0.
\end{align}
Fix $N>0$.  Using the Duhamel formula from Proposition~\ref{P:duhamel} together with the Strichartz inequality, the decomposition $u=u_{\leq N} + u_{> N}$, Lemma~\ref{L:neg deriv}, Lemma~\ref{L:product rule},
\eqref{good Strich}, Bernstein, H\"older, and Sobolev embedding, we find
\begin{align*}
\|\nabla u_{\leq N} \|_{L_t^\infty L_x^2}
&\lesssim \|\nabla P_{\leq N} F(u)\|_{L_t^2 L_x^{4/3}}\\
&\lesssim \|\nabla P_{\leq N}  F(u_{\leq N})\|_{L_t^2 L_x^{4/3}} + \|\nabla P_{\leq N} \O(u_{> N} u^2)\|_{L_t^2 L_x^{4/3}}\\
&\lesssim \|\nabla u_{\leq N}\|_{L_t^2 L_x^4} \|u_{\leq N}\|_{L_t^\infty L_x^4}^2 + N^{5/3} \bigl\||\nabla|^{-2/3} \O(u_{> N} u^2)\bigr\|_{L_t^2 L_x^{4/3}}\\
&\lesssim_u N^{3/2} + N^{5/3}  \bigl\||\nabla|^{-2/3} u_{> N}\bigr\|_{L_t^2 L_x^4} \bigl\||\nabla|^{2/3}u\bigr\|_{L_t^\infty L_x^{12/5}}\|u\|_{L_t^\infty L_x^4}\\
&\lesssim_u N^{3/2} + N^{5/3}  \sum_{M> N} M^{-5/3} \|\nabla u_M\bigr\|_{L_t^2 L_x^4}\\
&\lesssim_u N^{3/2}  + N^{5/3}  \sum_{M> N} M^{-1/6}\\
&\lesssim_u N^{3/2}.
\end{align*}
All spacetime norms in the estimates above are on $[0, \Tmax)\times\R^4$.

With \eqref{good energy} in place, we are now ready to finish the proof of Theorem~\ref{T:no cascade}.  First note that by Bernstein's inequality and \eqref{good energy}, $u\in L_t^\infty \dot H^{-1/4}_x([0, \Tmax)\times\R^4)$; indeed,
\begin{align*}
\bigl\||\nabla|^{-1/4} u\|_{L_t^\infty L_x^2}&\lesssim \bigl\||\nabla|^{-1/4} u_{> 1}\|_{L_t^\infty L_x^2} +\bigl\||\nabla|^{-1/4} u_{\leq 1}\|_{L_t^\infty L_x^2}\lesssim_u \sum_{N> 1} N^{-5/4} + \sum_{N\leq 1} N^{1/4}\lesssim_u1.
\end{align*}
Now fix $t\in [0, \Tmax)$ and let $\eta>0$ be a small constant.  By  Remark~\ref{R:small freq}, there exists $c(\eta)>0$ such that
$$
\int_{|\xi|\leq c(\eta)N(t)} |\xi|^2 |\hat u(t,\xi)|^2\, d\xi\leq \eta.
$$
Interpolating with $u\in L_t^\infty \dot H^{-1/4}_x$, we find
\begin{align}\label{close}
\int_{|\xi|\leq c(\eta)N(t)} |\hat u(t,\xi)|^2\, d\xi\lesssim_u \eta^{1/5}.
\end{align}
Meanwhile, by elementary considerations,
\begin{align}\label{far}
\int_{|\xi|\geq c(\eta)N(t)} |\hat u(t,\xi)|^2\, d\xi
\leq [c(\eta)N(t)]^{-2} \int_{\R^4} |\xi|^2 |\hat u(t,\xi)|^2\, d\xi \lesssim_u [c(\eta)N(t)]^{-2}.
\end{align}
Collecting \eqref{close} and \eqref{far} and using Plancherel's theorem, we obtain
$$
0\leq M(u(t)):=\int_{\R^4} |u(t,x)|^2\, dx \lesssim_u  \eta^{1/5} + c(\eta)^{-2} N(t)^{-2}
$$
for all $t\in [0, \Tmax)$.  Letting $\eta$ tend to zero and invoking \eqref{N blowup} and the conservation of mass, we conclude $M(u)=0$ and hence $u$ is identically zero.
This contradicts $\|u\|_{L_{t,x}^6( [0, \Tmax)\times\R^4)}=\infty$, thus settling Theorem~\ref{T:no cascade}.
\end{proof}

%%%%%%%%%%%%%%%%%%%%%%%%%%%%%%%%%%%%%%%%%%%%%%%%%%%%%%%%%%%%%%%%%%%%%%%%%%%%%%%%%%%%%%%%%%%
%
%
%                                   Section
%
%
%%%%%%%%%%%%%%%%%%%%%%%%%%%%%%%%%%%%%%%%%%%%%%%%%%%%%%%%%%%%%%%%%%%%%%%%%%%%%%%%%%%%%%%%%%%

\section{The quasi-soliton scenario}\label{S:no soliton}

In this section, we preclude the existence of almost periodic solutions as in Theorem~\ref{T:enemies} for which $\int_0^{\Tmax}N(t)^{-1}\,dt=\infty$.  We will show their existence is inconsistent with the interaction Morawetz inequality.

We start by recalling the interaction Morawetz inequality in four spatial dimensions; for details, see \cite{RV}.  For a solution $\phi:I\times\R^4\to \C$ to the equation $i\phi_{t}+\Delta \phi=\mathcal{N}$, we define the interaction Morawetz action
$$
M(t):=2\Im \int_{\R^4} \int_{\R^4} |\phi(t,y)|^{2}\frac{x-y}{|x-y|}\nabla\phi(t,x)\overline{\phi(t,x)}\, dx\, dy.
$$
Standard computations show
\begin{align*}
\partial_{t}M(t)& \geq 3\int_{\R^4} \int_{\R^4} \frac{|\phi(t,x)|^{2}|\phi(t,y)|^{2}}{|x-y|^{3}}\, dx\, dy + 4\Im \int_{\R^4} \int_{\R^4}\{\mathcal{N},\phi\}_{m}(t,y)\frac{x-y}{|x-y|}\nabla\phi(t,x)\overline{\phi(t,x)} \, dx\, dy \\
 &\quad+2\int_{\R^4} \int_{\R^4} |\phi(t,y)|^{2}\frac{x-y}{|x-y|} \{\mathcal{N},\phi\}_{p}(t,x)\,dx\,dy,
\end{align*}
where the mass bracket is given by $\{\mathcal N,\phi\}_m:=\Im(\mathcal N \bar{\phi})$ and the momentum bracket is given by $\{\mathcal N,\phi\}_p:=\Re(\mathcal N\nabla\overline{\phi}-\phi\nabla\overline{\mathcal N})$.
Thus, integrating with respect to time, we obtain
\begin{proposition}[Interaction Morawetz inequality] \label{P:im}
\begin{align*}
3&\int_I  \int_{\R^4} \int_{\R^4} \frac{|\phi(t,x)|^{2}|\phi(t,y)|^{2}}{|x-y|^{3}}\, dx\, dy\,dt+2\int_I\int_{\R^4} \int_{\R^4} |\phi(t,y)|^{2}\frac{x-y}{|x-y|} \{\mathcal{N},\phi\}_{p}(t,x)\,dx\,dy\,dt \\
&\leq 2\|\phi\|_{L_{t}^{\infty}L_{x}^{2}(I\times\R^{4})}^{3} \|\phi\|_{L_{t}^{\infty}\dot{H}_{x}^{1}(I\times\R^{4})} + 4\|\phi\|_{L_{t}^{\infty}L_{x}^{2}(I\times\R^{4})} \|\phi\|_{L_{t}^{\infty}\dot{H}_{x}^{1}(I\times\R^{4})} 
		\| \{\mathcal{N},\phi\}_{m}\|_{L_{t,x}^1(I\times\R^4)}.
\end{align*}
\end{proposition}

We will apply Proposition \ref{P:im} with $\phi=u_{\geq N}$ and $\mathcal{N}=P_{\geq N}(|u|^{2}u)$ for $N$ small enough that the Littlewood--Paley projection captures most of the solution.  More precisely, we will prove

\begin{proposition}[Frequency-localized interaction Morawetz estimate] \label{P:flim}
Let $u:[0, \Tmax)\times\R^4\to \C$ be an almost periodic solution to \eqref{nls} such that $N(t)\equiv N_k \geq 1$ on each characteristic interval $J_k\subset [0, \Tmax)$.  Then for any $\eta>0$ there exists $N_0=N_0(\eta)$ such that for $N\leq N_0$ and any compact time interval $I\subset [0, \Tmax)$, which is a union of contiguous subintervals $J_k$, we have 
\begin{align*}
\int_I\int_{\R^{4}}\int_{\R^{4}}\frac{|u_{\geq N}(t,x)|^{2}|u_{\geq N}(t,y)|^{2}}{|x-y|^{3}}\,dx\,dy\,dt\lesssim_u \eta \Bigl[N^{-3}+\int_I N(t)^{-1}\, dt\Bigr].
\end{align*}
The implicit constant does not depend on the interval $I$.
\end{proposition}

\begin{proof}
Fix a compact interval $I\subset [0, \Tmax)$, which is a union of contiguous subintervals $J_k$, and let $K:=\int_I N(t)^{-1}\, dt$.  Throughout the proof, all spacetime norms will be on $I\times\R^4$.

Fix $\eta>0$ and let $N_0=N_0(\eta)$ be small enough that claim \eqref{LTS:low} of Corollary~\ref{C:LTS} holds; more precisely, for all $N\leq N_0$,
\begin{align}\label{LTS: low'}
\|\nabla u_{\leq N}\|_{L_t^q L_x^r } \lesssim_u \eta(1+N^3K)^{1/q} \quad \text{for all} \quad \tfrac1q+\tfrac2r=1 \quad \text{with} \quad 2\leq q\leq \infty.
\end{align}
Choosing $N_0$ even smaller if necessary, we can also guarantee that
\begin{align}\label{high mass}
\|u_{\geq N}\|_{L_t^\infty L_x^2}\lesssim_u \eta^6 N^{-1} \quad \text{for all}\quad N\leq N_0.
\end{align}

Now fix $N\leq N_0$ and write $\ulo:=u_{\leq N}$ and $\uhi:=u_{> N}$.  With this notation, \eqref{LTS: low'} becomes
\begin{align}\label{LTS:lo}
\|\nabla \ulo \|_{L_t^q L_x^r } \lesssim_u \eta (1+N^3K)^{1/q} \quad \text{for all} \quad \tfrac1q+\tfrac2r=1 \quad \text{with} \quad2\leq q\leq \infty.
\end{align}
We will also need claim \eqref{LTS:high} of Corollary~\ref{C:LTS}, which reads
\begin{align}\label{LTS:hi}
\| \uhi \|_{L_t^q L_x^r } \lesssim_u N^{-1}(1+N^3K)^{1/q} \quad \text{for all} \quad \tfrac1q+\tfrac2r=1\quad \text{with} \quad 3< q\leq \infty.
\end{align}
Note that by \eqref{high mass}, the endpoint $q=\infty$ of the inequality above is strengthened to
\begin{align}\label{hi mass}
\|\uhi\|_{L_t^\infty L_x^2}\lesssim_u \eta^6 N^{-1}.
\end{align}

To continue, we apply Proposition~\ref{P:im} with $\phi=\uhi$ and $\mathcal{N}=\ph F(u)$ and use \eqref{hi mass}:
\begin{align}\label{to estimate}
\int_I\int_{\R^{4}}\int_{\R^{4}}&\frac{|\uhi(t,x)|^{2}|\uhi(t,y)|^{2}}{|x-y|^{3}}\,dx\,dy\,dt \notag\\
& + \int_I\int_{\R^4} \int_{\R^4} |\uhi(t,y)|^{2}\frac{x-y}{|x-y|} \{\ph F(u),\uhi\}_{p}(t,x)\,dx\,dy\,dt \\
&\qquad \qquad \qquad \lesssim_u \eta^{18}N^{-3} + \eta^6 N^{-1} \| \{\ph F(u),\uhi\}_{m}\|_{L_{t,x}^1(I\times\R^4)}. \notag
\end{align}
We first consider the contribution of the momentum bracket term.  We write
\begin{align*}
\{\ph F(u),\uhi\}_{p}
&=\{F(u),u\}_p- \{F(\ulo),\ulo\}_p - \{F(u)-F(\ulo), \ulo\}_p - \{\pl F(u), \uhi\}_p\\
&=-\tfrac12\nabla[ |u|^4-|\ulo|^4]  - \{F(u)-F(\ulo), \ulo\}_p - \{\pl F(u), \uhi\}_p\\
&=:I+II+III.
\end{align*}
After an integration by parts, the term $I$ contributes to the left-hand side of  \eqref{to estimate} a multiple of
\begin{align*}
\int_I\int_{\R^4} \int_{\R^4} &\frac{|u_{hi}(t,y)|^{2}|u_{hi}(t,x)|^{4}}{|x-y|}\,dx\,dy\,dt +\sum_{j=1}^{3}\Bigl|\int_I\int_{\R^4} \int_{\R^4} \frac{|\uhi(t,y)|^{2}\O(u_{hi}^ju_{lo}^{4-j})(t,x)}{|x-y|}\,dx\,dy\,dt\Bigr|,
\end{align*}
In order to estimate the contribution of $II$ to \eqref{to estimate}, we use $\{f,g\}_{p}=\nabla \O(fg)+\O(f\nabla g)$ to write
$$
 \{F(u)-F(\ulo), \ulo\}_p=\sum_{j=1}^{3} \nabla\O(u_{hi}^{j}u_{lo}^{4-j}) +\sum_{j=1}^{3} \O(u_{hi}^ju_{lo}^{3-j}\nabla \ulo).
$$
Integrating by parts for the first term and bringing absolute values inside the integrals for the second term, we find that $II$ contributes to the right-hand side of \eqref{to estimate} a multiple of
\begin{align*}
\sum_{j=1}^{3}\int_I\int_{\R^4} \int_{\R^4} & \frac{|\uhi(t,y)|^{2}|\uhi(t,x)|^j|\ulo(t,x)|^{4-j}}{|x-y|}\,dx\,dy\,dt \\
&+ \sum_{j=1}^{3}\int_I\int_{\R^4} \int_{\R^4} |\uhi(t,y)|^{2}|\uhi(t,x)|^j|\nabla \ulo(t,x)||\ulo(t,x)|^{3-j}\,dx\,dy\,dt.
\end{align*}
Finally, integrating by parts when the derivative (from the definition of the momentum bracket) falls on $\uhi$, we estimate the contribution of $III$ to the right-hand side of \eqref{to estimate} by a multiple of
\begin{align*}
\int_I \int_{\R^4} \int_{\R^4} &\frac{|\uhi(t,y)|^{2}|\uhi(t,x)||\pl F(u(t,x))|}{|x-y|}\,dx\,dy\,dt \\
&+ \int_I\int_{\R^4} \int_{\R^4} |\uhi(t,y)|^{2}|\uhi(t,x)| |\nabla \pl F(u(t,x))|\,dx\,dy\,dt.
\end{align*}

Consider now the mass bracket appearing in \eqref{to estimate}.  Exploiting cancellation, we write
\begin{align*}
\{ & \ph  F(u), \uhi\}_m\\
&= \{\ph F(u) - F(\uhi), \uhi\}_m \\
&=  \{\ph \bigl[ F(u) - F(\uhi) - F(\ulo)\bigr] , \uhi\}_m + \{\ph F(\ulo), \uhi\}_m  - \{\pl F(\uhi) , \uhi\}_m\\
&=\O(u_{hi}^3\ulo) + \O(u_{hi}^2u_{lo}^2) + \{\ph F(\ulo), \uhi\}_m  - \{\pl F(\uhi) , \uhi\}_m.
\end{align*}

Putting everything together and using \eqref{hi mass}, \eqref{to estimate} becomes
\begin{align}
\int_I& \int_{\R^{4}}\int_{\R^{4}}\frac{|\uhi(t,x)|^{2}|\uhi(t,y)|^{2}}{|x-y|^{3}}\,dx\,dy\,dt +\int_I\int_{\R^{4}}\int_{\R^{4}}\frac{|\uhi(t,x)|^{2}|\uhi(t,y)|^4}{|x-y|}\,dx\,dy\,dt\label{e0}\\
&\lesssim_u \eta^{18}N^{-3} \label{e1}\\
&\quad +\eta^6 N^{-1}\bigl\{ \|u_{hi}^3\ulo\|_{L_{t,x}^1} +   \|u_{hi}^2u_{lo}^2\|_{L_{t,x}^1} +  \|\uhi\ph F(\ulo)\|_{L_{t,x}^1} + \| \uhi \pl F(\uhi)\|_{L_{t,x}^1} \bigr\} \label{e2}\\
&\quad + \eta^{12} N^{-2} \sum_{j=1}^{3} \|u_{hi}^j u_{lo}^{3-j}\nabla \ulo\|_{L_{t,x}^1} + \eta^{12} N^{-2}\|\uhi \nabla \pl F(u)\|_{L_{t,x}^1}\label{e3}\\
&\quad + \sum_{j=1}^{3} \int_I\int_{\R^4} \int_{\R^4}  \frac{|\uhi(t,y)|^{2}|\uhi(t,x)|^j|\ulo(t,x)|^{4-j}}{|x-y|}\,dx\,dy\,dt \label{e4}\\
&\quad + \int_I \int_{\R^4} \int_{\R^4} \frac{|\uhi(t,y)|^{2}|\uhi(t,x)||\pl F(u(t,x))|}{|x-y|}\,dx\,dy\,dt.\label{e5}
\end{align}
Thus, to complete the proof of Proposition~\ref{P:flim} we have to show that the error terms \eqref{e2} through \eqref{e5} are acceptable; clearly, \eqref{e1} is acceptable.

Consider now error term \eqref{e2}.  Using \eqref{LTS:lo}, \eqref{LTS:hi}, and Sobolev embedding, we estimate
\begin{align*}
\|u_{hi}^3\ulo\|_{L_{t,x}^1}&\lesssim \|\uhi\|_{L_t^\infty L_x^4} \|\uhi\|_{L_t^{7/2}L_x^{14/5}}^2\|\ulo\|_{L_t^{7/3}L_x^{28}}\lesssim_u\eta N^{-2}(1+N^3K)\\
\|u_{hi}^2u_{lo}^2\|_{L_{t,x}^1}&\lesssim \|\uhi\|_{L_t^4 L_x^{8/3}}^2 \|\ulo\|_{L_t^4L_x^8}^2\lesssim_u\eta^2 N^{-2}(1+N^3K).
\end{align*}
Using Bernstein's inequality as well, we estimate
\begin{align*}
\|\uhi \ph F(\ulo)\|_{L_{t,x}^1}
&\lesssim\|\uhi\|_{L_t^4 L_x^{8/3}} N^{-1} \|\nabla F(\ulo)\|_{L_t^{4/3}L_x^{8/5}}\\
&\lesssim_u N^{-2} (1+N^3K)^{1/4} \|\nabla \ulo\|_{L_t^2L_x^4} \|\ulo\|_{L_t^8L_x^{16/3}}^2\\
&\lesssim_u\eta^3 N^{-2}(1+N^3K).
\end{align*}
Finally, by H\"older, Bernstein, Sobolev embedding, \eqref{LTS:lo} and \eqref{LTS:hi},
\begin{align*}
\| \uhi\pl F(\uhi)  \|_{L_{t,x}^1} 
&\lesssim \|\uhi\|_{L_t^{10/3} L_x^{20/7}} N^{7/5} \|F(\uhi)\|_{L_t^{10/7} L_x^1}\\
&\lesssim_u N^{2/5}(1+N^3K)^{3/10} \|\uhi\|_{L_t^{10/3} L_x^{20/7}}^{7/3} \|\uhi\|_{L_t^\infty L_x^{40/11}}^{2/3}\\
&\lesssim_u N^{2/5 - 7/3}(1+N^3K) \| |\nabla|^{9/10} \uhi\|_{L_t^\infty L_x^2}^{2/3}\\
&\lesssim_u N^{-2}(1+N^3K).
\end{align*}
Collecting the estimates above we find
$$
\eqref{e2}\lesssim_u \eta^6 N^{-3}(1+N^3K)\lesssim_u \eta(N^{-3}+K),
$$
and thus this error term is acceptable.

Consider next error term \eqref{e3}.  By \eqref{LTS:lo}, \eqref{LTS:hi}, \eqref{hi mass}, Sobolev embedding, and Bernstein,
\begin{align*}
\|u_{hi}u_{lo}^2\nabla \ulo\|_{L_{t,x}^1}&\lesssim \|\nabla \ulo\|_{L_t^2L_x^4} \|\uhi\|_{L_t^\infty L_x^2} \|\ulo\|_{L_t^4L_x^8}^2 \lesssim_u\eta^9 N^{-1}(1+N^3K)\\
\|u_{hi}^2u_{lo}\nabla \ulo\|_{L_{t,x}^1}&\lesssim \|\nabla \ulo\|_{L_t^2L_x^4} \|\uhi\|_{L_t^4 L_x^{8/3}}^2 \|\ulo\|_{L_{t,x}^\infty}\lesssim_u\eta^2 N^{-1}(1+N^3K)\\
\|u_{hi}^3\nabla \ulo\|_{L_{t,x}^1}&\lesssim \|\nabla \ulo\|_{L_t^{7/3}L_x^{28}} \|\uhi\|_{L_t^{7/2}L_x^{14/5}}^2 \|\uhi\|_{L_t^\infty L_x^4} \lesssim_u\eta N^{-1}(1+N^3K).
\end{align*}
To estimate the second term in \eqref{e3}, we write $F(u)=F(\ulo) + \O(\uhi u_{lo}^2 + u_{hi}^2\ulo +u_{hi}^3)$.  Arguing as above, we obtain
\begin{align*}
\|\uhi \nabla \pl F(\ulo)\|_{L_{t,x}^1}&\lesssim \|\uhi\|_{L_t^\infty L_x^2}\|\nabla \ulo\|_{L_t^2L_x^4} \|\ulo\|_{L_t^4L_x^8}^2 \lesssim_u\eta^9 N^{-1}(1+N^3K)\\
\| \uhi \nabla \pl \O(\uhi u_{lo}^2) \|_{L_{t,x}^1}&\lesssim  N \|\uhi\|_{L_t^4 L_x^{8/3}}^2 \|\ulo\|_{L_t^4L_x^8}^2\lesssim_u\eta^2 N^{-1}(1+N^3K)\\
\|\uhi \nabla \pl \O( u_{hi}^2\ulo) \|_{L_{t,x}^1}&\lesssim N \|\uhi\|_{L_t^\infty L_x^4} \|\uhi\|_{L_t^{7/2}L_x^{14/5}}^2\|\ulo\|_{L_t^{7/3}L_x^{28}}\lesssim_u\eta N^{-1}(1+N^3K)\\
\|\uhi \nabla \pl \O( u_{hi}^3) \|_{L_{t,x}^1}&\lesssim  \|\uhi\|_{L_t^{10/3} L_x^{20/7}} N^{12/5} \|u_{hi}^3\|_{L_t^{10/7} L_x^1}\lesssim N^{12/5}  \|\uhi\|_{L_t^{10/3} L_x^{20/7}}^{10/3} \|\uhi\|_{L_t^\infty L_x^{40/11}}^{2/3}\\
&\lesssim_u N^{-1}(1+N^3K).
\end{align*}
Putting everything together, we find
$$
\eqref{e3}\lesssim_u \eta^{12} N^{-3}(1+N^3K)\lesssim_u \eta(N^{-3}+K),
$$
and thus this error term is also acceptable.

We now turn to error term \eqref{e4}.  By trivial considerations, we only have to consider the cases $j=1$ and $j=3$.  We start with the case $j=1$; using H\"older together with the Hardy--Littlewood--Sobolev inequality,
Sobolev embedding, \eqref{LTS:lo}, \eqref{LTS:hi}, and \eqref{hi mass}, we estimate
\begin{align*}
 \int_I\int_{\R^4} \int_{\R^4}  \frac{|\uhi(t,y)|^{2}|\uhi(t,x)||\ulo(t,x)|^3}{|x-y|}\,dx\,dy\,dt
&\lesssim \|\uhi\|_{L_t^{12}L_x^{24/11}}^2 \Bigl\| \tfrac1{|x|}* \bigl(|\uhi||\ulo|^3\bigr)\Bigr\|_{L_t^{6/5}L_x^{12}}\\
&\lesssim_u N^{-2} (1+N^3K)^{1/6}\|\uhi u_{lo}^3\|_{L_{t,x}^{6/5}}\\
&\lesssim_u N^{-2} (1+N^3K)^{1/6} \|\uhi\|_{L_t^\infty L_x^2} \|\ulo\|_{L_t^{18/5} L_x^9}^3\\
&\lesssim_u \eta^9N^{-3} (1+N^3K).
\end{align*}
Finally, to estimate the error term corresponding to $j=3$, we consider two scenarios:  If $|\ulo|\leq 10^{-6}|\uhi|$, we absorb this contribution into the term
$$
\int_I\int_{\R^{4}}\int_{\R^{4}}\frac{|\uhi(t,x)|^{2}|\uhi(t,y)|^4}{|x-y|}\,dx\,dy\,dt,
$$
which appears in \eqref{e0}.  If instead $|\uhi|\leq 10^6|\ulo|$, we may estimate the contribution of this term by that of the error term corresponding to $j=1$.  Thus,
$$
\eqref{e4}\lesssim_u \eta^9 N^{-3}(1+N^3K)\lesssim_u \eta(N^{-3}+K),
$$
which renders this error term acceptable.

We are left to consider error term \eqref{e5}.  Arguing as for the case $j=1$ of the error term \eqref{e4}, we derive 
\begin{align*}
\int_I \int_{\R^4} \int_{\R^4}& \frac{|\uhi(t,y)|^{2}|\uhi(t,x)||\pl F(u(t,x))|}{|x-y|}\,dx\,dy\,dt\\
&\qquad \lesssim \|\uhi\|_{L_t^{12}L_x^{24/11}}^2 \Bigl\| \tfrac1{|x|}* \bigl(|\uhi||\pl F(u)|\bigr)\Bigr\|_{L_t^{6/5}L_x^{12}}\\
&\qquad\lesssim_u N^{-2} (1+N^3K)^{1/6}\|\uhi \pl F(u)\|_{L_{t,x}^{6/5}}\\
&\qquad\lesssim_u N^{-2} (1+N^3K)^{1/6} \|\uhi\|_{L_t^4L_x^{8/3}} \|\pl F(u)\|_{L_t^{12/7} L_x^{24/11}}\\
&\qquad\lesssim_u N^{-3} (1+N^3K)^{5/12}\|\pl F(u)\|_{L_t^{12/7} L_x^{24/11}}.
\end{align*}
We now write $F(u)=F(\uhi) + \O(u_{lo}^3 + u_{lo}^2\uhi +\ulo u_{hi}^2)$.  Using H\"older, Bernstein, Sobolev embedding, \eqref{LTS:lo}, \eqref{LTS:hi}, and \eqref{hi mass}, we estimate
\begin{align*}
\|\pl \O(u_{lo}^3)\|_{L_t^{12/7} L_x^{24/11}}&\lesssim \|\ulo\|_{L_t^{12} L_x^{24/5}}\|\ulo\|_{L_t^4L_x^8}^2\lesssim_u\eta^3 (1+N^3K)^{7/12}\\
\|\pl \O(u_{lo}^2\uhi)\|_{L_t^{12/7} L_x^{24/11}}&\lesssim N \|u_{lo}^2\uhi\|_{L_t^{12/7} L_x^{24/17}}\lesssim N \|\ulo\|_{L_t^4L_x^8}^2 \|\uhi\|_{L_t^{12} L_x^{24/11}}\lesssim_u\eta^2 (1+N^3K)^{\frac7{12}}\\
\|\pl \O(\ulo u_{hi}^2)\|_{L_t^{12/7} L_x^{24/11}}&\lesssim N \|\ulo u_{hi}^2\|_{L_t^{12/7} L_x^{24/17}}\lesssim N\|\ulo\|_{L_t^3L_x^{12}}\|\uhi\|_{L_t^4 L_x^{8/3}} \|\uhi\|_{L_t^\infty L_x^4}\\
		&\lesssim_u \eta (1+N^3K)^{7/12}.
\end{align*}
Finally, using Bernstein, H\"older, interpolation, \eqref{LTS:lo}, \eqref{LTS:hi}, and \eqref{hi mass}, we get
\begin{align*}
\|\pl F(\uhi)\|_{L_t^{12/7} L_x^{24/11}}
&\lesssim N^{13/6}\| F(\uhi)\|_{L_t^{12/7} L_x^1}\lesssim N^{13/6} \|\uhi\|_{L_t^{24/7} L_x^{48/17}}^2 \|\uhi\|_{L_t^\infty L_x^{24/7}}\\
&\lesssim_u N^{1/6}(1+N^3K)^{7/12} \| |\nabla|^{5/6}\uhi\|_{L_t^\infty L_x^2}\\
&\lesssim_u \eta (1+N^3K)^{7/12}.
\end{align*}
Collecting these estimates, we find
$$
\eqref{e5}\lesssim_u \eta N^{-3}(1+N^3K)\lesssim_u \eta(N^{-3}+K),
$$
and thus this error term is also acceptable.

This concludes the proof of Proposition~\ref{P:flim}.
\end{proof}

With Proposition~\ref{P:flim} in place, we are now ready to preclude the second (and last) scenario of Theorem~\ref{T:enemies} and thus complete the proof of Theorem~\ref{T:main}.

\begin{theorem}[No quasi-solitons]\label{T:no soliton}
There are no almost periodic solutions $u:[0, \Tmax)\times\R^4\to \C$ to \eqref{nls} with $N(t)\equiv N_k \geq 1$ on each characteristic interval $J_k\subset [0, \Tmax)$
which satisfy $\|u\|_{L^6_{t,x}( [0, \Tmax) \times \R^4)} =+\infty$ and
\begin{align}\label{infinite int}
\int_0^{\Tmax}N(t)^{-1}\,dt=\infty.
\end{align}
\end{theorem}

\begin{proof}
We argue by contradiction.  Assume there exists such a solution $u$.

Let $\eta>0$ be a small parameter to be chosen later.  By Proposition~\ref{P:flim}, there exists $N_0=N_0(\eta)$ such that for all $N\leq N_0$ and any compact time interval $I\subset [0, \Tmax)$, which is a union of contiguous subintervals $J_k$, we have 
\begin{align}\label{flim to use}
\int_I\int_{\R^{4}}\int_{\R^{4}}\frac{|u_{\geq N}(t,x)|^{2}|u_{\geq N}(t,y)|^{2}}{|x-y|^{3}}\,dx\,dy\,dt\lesssim_u \eta \Bigl[N^{-3}+\int_I N(t)^{-1}\, dt\Bigr].
\end{align}
As $\inf_{t\in [0, \Tmax)}N(t)\geq 1$, choosing $N_0$ even smaller if necessary (depending on $\eta$) we can also ensure that
\begin{align}\label{E:small 1}
\|u_{\leq N} \|_{L_t^\infty L_x^4 ([0, \Tmax)\times\R^4)} + \|\nabla u_{\leq N}\|_{L_t^\infty L_x^2 ([0, \Tmax)\times\R^4)}\leq \eta \quad \text{for all} \quad N\leq N_0.
\end{align}

Next we claim that there exists $C(u)>0$ such that
\begin{align}\label{concentration}
N(t)^2 \int_{|x-x(t)|\leq C(u)/N(t)} |u(t,x)|^2\, dx \gtrsim_u 1/C(u)
\end{align}
uniformly for $t\in [0, \Tmax)$.  That this is true for a single time $t$ follows from the fact that $u(t)$ is not identically zero.  To upgrade this to a statement uniform in time, we use the fact that $u$ is almost periodic.  More precisely, we note that the left-hand side of \eqref{concentration} is scale invariant and that the map $u(t)\mapsto \text{LHS} \eqref{concentration}$ is continuous on $L_x^4$ and hence also on $\dot H^1_x$.  

Next, using H\"older's inequality and \eqref{E:small 1}, we find
\begin{align*}
\int_{|x-x(t)|\leq C(u)/N(t)} |u_{\leq N}(t,x)|^2\, dx \lesssim \Bigl\{ \frac{C(u)}{N(t)} \|u_{\leq N}\|_{L_t^\infty L_x^4 ([0, \Tmax)\times\R^4)} \Bigr\}^2\lesssim_u \eta^2C(u)^2 N(t)^{-2}
\end{align*}
for all $t\in [0, \Tmax)$ and all $N\leq N_0$.  Combining this with \eqref{concentration} and choosing $\eta$ sufficiently small depending on $u$, we find
\begin{align*}
\inf_{t\in [0, \Tmax)} N(t)^2 \int_{|x-x(t)|\leq C(u)/N(t)} |u_{\geq N} (t,x)|^2\, dx \gtrsim_u 1  \quad \text{for all} \quad N\leq N_0.
\end{align*}
Thus, on any compact time interval $I\subset [0, \Tmax)$ and for any $N\leq N_0$ we have
\begin{align*}
\int_I\int_{\R^{4}}& \int_{\R^{4}}\frac{|u_{\geq N}(t,x)|^{2}|u_{\geq N}(t,y)|^{2}}{|x-y|^{3}}\,dx\,dy\,dt\\
&\geq \int_I \iint_{|x-y|\leq \frac{2C(u)}{N(t)}} \Bigl[\frac{N(t)}{2C(u)}\Bigr]^{3}|u_{\geq N}(t,x)|^{2}|u_{\geq N}(t,y)|^{2}\,dx\,dy\,dt\\
&\geq \int_I\Bigl[\frac{N(t)}{2C(u)}\Bigr]^{3}  \int_{|x-x(t)|\leq \frac{C(u)}{N(t)}} |u_{\geq N}(t,x)|^{2} \, dx \int_{|y-x(t)|\leq \frac{C(u)}{N(t)}} |u_{\geq N}(t,y)|^{2}\,dy\,dt\\
&\gtrsim_u  \int_I N(t)^{-1}\, dt.
\end{align*}
Invoking \eqref{flim to use} and choosing $\eta$ small depending on $u$, we find
$$
\int_I N(t)^{-1}\, dt \lesssim_u N^{-3} \qquad \text{for all }  I\subset [0, \Tmax) \text{ and all } N\leq N_0.
$$ 
Recalling the hypothesis \eqref{infinite int}, we derive a contradiction by choosing the interval $I$ sufficiently large inside $[0, \Tmax)$.  

This completes the proof of the theorem.
\end{proof}

%%%%%%%%%%%%%%%%%%%%%%%%%%%%%%%%%%%%%%%%%%%%%%%%%%%%%%%%%%%%%%%%%%%%%%%%%%%%%%%%%%%%%%%%%%%
%
%
%                                   Section
%
%
%%%%%%%%%%%%%%%%%%%%%%%%%%%%%%%%%%%%%%%%%%%%%%%%%%%%%%%%%%%%%%%%%%%%%%%%%%%%%%%%%%%%%%%%%%%

\end{document}